\documentclass[11pt,reqno]{amsproc}
\title[]{Stability of Weak Electrokinetic Flow}

\author{Fizay-Noah Lee}
\address{Department of Mathematics, Vanderbilt University, Nashville, TN 37235}
\email{noah.lee@vanderbilt.edu}
\usepackage[margin=1in]{geometry}
\usepackage{amsmath, amsthm, amssymb}
\usepackage{times}
\usepackage{color}
\usepackage{hyperref}
\usepackage{comment}
\usepackage{enumerate}
\usepackage{setspace}
\usepackage{xcolor}

\newcommand{\pa}{\partial}

\newcommand{\la}{\label}
\newcommand{\fr}{\frac}
\newcommand{\na}{\nabla}
\newcommand{\be}{\begin{equation}}
\newcommand{\ee}{\end{equation}}
\newcommand{\bal}{\begin{aligned}}
\newcommand{\eal}{\end{aligned}}
\newcommand{\ba}{\begin{array}{l}}
\newcommand{\ea}{\end{array}}
\newcommand{\Rr}{{\mathbb R}}

\newtheorem{thm}{Theorem}
\newtheorem{prop}{Proposition}

\newtheorem{rem}{Remark}

\renewcommand{\div}{{\mbox{div}\,}}
\newcommand{\D}{\Delta}

\date{today}

\begin{document}

\noindent

\begin{abstract}
    We consider the Nernst-Planck-Stokes system on a bounded domain of $\Rr^d$, $d=2,3$ with general nonequilibrium Dirichlet boundary conditions for the ionic concentrations. It is well known that, in a wide range of cases, equilibrium steady state solutions of the system, characterized by zero fluid flow, are asymptotically stable. In these regimes, the existence of a natural dissipative structure is critical in obtaining stability. This structure, in general, breaks down under nonequilibrium conditions, in which case, in the steady state, the fluid flow may be nontrivial. In this short paper, we show that, nonetheless, certain classes of very weak nonequilibrium steady states, with nonzero fluid flow, remain globally asymptotically stable.
\end{abstract}

\maketitle

\section{Introduction}
We consider the Nernst-Planck-Stokes (NPS) system in a connected, but not necessarily simply connected bounded domain $\Omega\subset\mathbb{R}^d$ ($d=2,3$) with smooth boundary. The system models electrodiffusion of ions in a fluid in the presence of an applied electrical potential on the boundary \cite{prob,rubibook}. In this paper, we consider the case where there are two, oppositely charged ionic species with, for simplicity, charges $+1$ and $-1$. Mathematically, the system is given by the Nernst-Planck equations
\be
\bal
\pa_t c_1+u\cdot\na c_1=&D_1\div(\na c_1+c_1\na\Phi)\\
\pa_t c_2+u\cdot\na c_2=&D_2\div(\na c_2-c_2\na\Phi)\la{np}
\eal
\ee
coupled to the Poisson equation
\be
-\epsilon\D\Phi=c_1-c_2=\rho\la{pois}
\ee
and to the Stokes system
\be
\pa_t u-\nu\D u+\na p=-K\rho\na\Phi,\quad \div u=0.\la{nse}
\ee
Above, $c_i$ are the local ionic concentrations respectively, $\rho$ is a rescaled local charge density, $u$ is the fluid velocity, and $\Phi$ is a rescaled electrical potential. The constant $K>0$ is a coupling constant given by the product of Boltzmann's constant $k_B$ and the absolute temperature $T_K$. The constants $D_i$ are the ionic diffusivities, $\epsilon>0$ is a rescaled dielectric permittivity of the solvent proportional to the square of the Debye length, and $\nu>0$ is the kinematic viscosity of the fluid. The dimensional counterparts of $\Phi$ and $\rho$ are given by $(k_BT_k/e)\Phi$ and $e\rho$, respectively, where $e$ is elementary charge.

It is well known that for certain \textit{equilibrium boundary conditions,} (see \eqref{eqc1} and  \eqref{eqc2} below) the NPS system (\ref{np})-(\ref{nse}) admits a unique steady state solution, with vanishing fluid velocity $u^* \equiv 0$, and with ionic concentrations $c_i^*$ explicitly related to $\Phi^*$, which uniquely solves a nonlinear Poisson equation, known as the Poisson-Boltzmann equation, 
\be
\bal
-\epsilon\D\Phi^* &= c_1^*-c_2^*\\
c_1^* &= \fr{e^{-\Phi^*}}{Z_1}\\
c_2^* &= \fr{e^{\Phi^*}}{Z_2}
\eal
\la{PB}
\ee
with normalizing constants $Z_1,Z_2>0$.

For equilibrium boundary conditions it is known that for $d=2$ the unique steady states are globally stable \cite{bothe,ci} and for $d=3$ locally stable \cite{np3d,ryham}. Equilibrium boundary conditions include the cases where $c_i$ obey blocking boundary conditions
\be\la{eqc1}
(\pa_nc_1+c_1\pa_n\Phi)_{|\pa\Omega}=(\pa_nc_2-c_2\pa_n\Phi)_{|\pa\Omega}=0
\ee
(here, $\pa_n$ is the normal derivative).

Another example of equilibrium boundary conditions is where $c_i$ obey a mix of blocking and Dirichlet boundary conditions and $\Phi$ obeys Dirichlet boundary conditions in such a way that the electrochemical potentials
\be
\mu_1=\log c_1+\Phi,\quad \mu_2=\log c_2-\Phi
\ee
are each constant on the boundary portions where $c_i$ obey Dirichlet boundary conditions. That is, if $c_i$ satisfy Dirichlet boundary conditions on $S_i\subset\pa\Omega$, then one must have
\be
{\mu_i}_{|S_i}=\text{constant}.\la{eqc2}
\ee
We point out that the latter condition is a very strong requirement, as one cannot in general pick $\Phi_{|\pa\Omega}$ that will satisfy \eqref{eqc2} for both $i=1$ and $i=2$ simultaneously, if ${c_1}_{|S_1}$ and ${c_2}_{|S_2}$ are arbitrarily prescribed. Such a choice would be possible if, for example, ${c_1}_{|S_1}$ and ${c_2}_{|S_2}$ were each chosen to be constant.

 In general, deviations from equilibrium boundary conditions can produce instabilities and even chaotic behavior for time dependent solutions of NPS \cite{davidson,kang,pham,rubinstein,rubizaltz,zaltzrubi}. Furthermore, the existence of steady states in nonequilibrium configurations is not guaranteed in full generality, and even in cases when existence is known, their uniqueness is, in general, not known to hold \cite{mock,park}.

In both the mathematical and physical literature, a wide range of boundary conditions are considered for both the ionic concentrations and the electrical potential, a few of which are mentioned above. We refer the reader to \cite{bothe,ci,cil,davidson,mock,schmuck} for further discussions on boundary conditions. In this paper, we consider Dirichlet boundary conditions for both $c_i$ and $\Phi$, together with no-slip boundary conditions for $u$,
\be
\bal
{c_i}_{|\pa\Omega}&=\gamma_i>0,\quad i=1,2\\
\Phi_{|\pa\Omega}&=W\\
u_{|\pa\Omega}&=0.\la{BC}
\eal
\ee

We assume $\gamma_i,W\in C^\infty(\pa\Omega)$, but we do not require them to be constant. For $c_i$, the Dirichlet boundary conditions model, for example, ion-selectivity at an ion-selective membrane or some fixed concentration of ions at the boundary layer-bulk interface. Dirichlet boundary conditions for $\Phi$ model an applied electric potential on the boundary. 

There is a large literature on the well-posedness of the time dependent NPS (and the related Nernst-Planck-Navier-Stokes) system \cite{bothe,ci,np3d,cil,fischer,fnl,liu,ryham,schmuck}, as well as the uncoupled Nernst-Planck \cite{biler,biler2,choi,gaj,gajewski,mock} and Navier-Stokes systems \cite{cf,temam}. 

For questions of long time behavior and stability of NPS solutions, most studies have focused on the case of equilibrium boundary conditions. As briefly mentioned above, the corresponding steady states in these cases are the unique Nernst-Planck steady states $c_i^*$, together with zero fluid flow $u^*\equiv 0$ c.f. \cite{bothe, ci, np3d, NPS}. Throughout this paper, we shall refer to such steady states as \textit{equilibrium steady states}.

Part of the reason for the plentiful results in the equilibrium regimes is the fact that for these cases, there generally exists a natural dissipative structure (manifested by a dissipative energy inequality), from which stability can be deduced. On the other hand in both experiments \cite{kang,rubinstein} and numerical simulations \cite{davidson, pham}, so called electrokinetic instabilities (EKI) are known to occur in electrochemical systems whereby electrically induced flow patterns emerge. These patterns range from vortical to chaotic, and the mechanism behind the emergence of these patterns have long been a subject of study \cite{rubizaltz}. 

In light of these observations, it is of interest to analytically study the dynamical properties of solutions of the NPS system in regimes where nontrivial flow patterns may emerge. In \cite{ltd} we establish the existence and finite dimensionality of a global attractor for the NPS system. The results hold for boundary conditions that, in particular, include cases where the relevant steady state solution(s) has nonzero fluid flow (c.f. Theorem \ref{boundarythm} below). The existence of a global attractor establishes stability in a rather generalized sense in that, loosely speaking, all time dependent solutions must converge towards a small subset of the underlying function space. While this analysis provides a partial quantitative and qualitative picture of the long time dynamics of NPS solutions, it does not directly address the question of the stability of individual steady states. 

On the other hand, as mentioned above, several works study the stability of specific steady state solutions, but these results are restricted to equilibrium state states \cite{bothe, ci, np3d, NPS}.

In this paper, we partially bridge the gap in the existing literature by proving that there exist stable, nonequilibrium steady state solutions of NPS with nonzero fluid flow. We note that one can prove analogous results for the Nernst-Planck-Navier-Stokes system; however, due to the difficulty associated with the nonlinearity in the Navier-Stokes equations \cite{cf}, some additional constraints must be added (e.g. restricting to two dimensions or assuming the regularity of the velocity $u$). We shall not pursue this generalization in this paper.

\subsection{Notation}
Unless otherwise stated, we denote by $C$ a positive constant that depends only on the parameters of the system and the domain, but not the boundary and initial conditions, unless otherwise stated. The value of $C$ may differ from line to line.

We denote by $L^p=L^p(\Omega)$ the Lebesgue spaces and by $W^{s,p}=W^{s,p}(\Omega)$, $H^s=H^s(\Omega)=W^{s,2}$ the Sobolev spaces. We also denote by $L^p(\pa\Omega)$ the Lebesgue spaces of functions defined on the boundary. In this latter case, the domain is explicitly indicated. 

We denote by $dx$ the volume element in $\Omega$ and by $dS$ the surface element on $\pa\Omega$.

\section{Preliminaries}\la{ssnpns}
Prior to studying the stability of specific steady state solutions of NPS, we discuss what is known about the steady state system. Theorem 1 of \cite{NPS} establishes the existence of a steady state solution of the Nernst-Planck-Stokes system, which we state below.
\begin{thm}\la{stex}
There exists a smooth solution of the steady state Nernst-Planck-Stokes system
\begin{align}
u\cdot\na c_1&= D_1\div (\na c_1+c_1\na\Phi)\la{s}\\
u\cdot\na c_2&= D_2\div (\na c_2-c_2\na\Phi)\\
-\epsilon\D\Phi&=c_1-c_2=\rho\\
-\nu\D u+\na p &= -K\rho\na\Phi\la{sstokes}\\
\div u&=0\la{S}
\end{align}
on a smooth, connected, bounded domain $\Omega\subset\mathbb{R}^d$ ($d=2,3)$ together with boundary conditions
\begin{align}
    {c_i}_{|\pa\Omega}&=\gamma_i>0,\quad i=1,2\la{b}\\
    \Phi_{|\pa\Omega}&=W\\
    u_{|\pa\Omega}&=0\la{B}
\end{align}
with $\gamma_i,W\in C^\infty(\pa\Omega)$.
\end{thm}
In the theorem above and hereafter, by ``solution'' to NPS, we always refer to physical solutions that satisfy $c_i\ge 0$.

In \cite{NPS}, the existence of a solution is obtained via a fixed point theorem, so few things can be said precisely about the solution thus obtained. However, below are some pointwise bounds that can be rigorously established.

 \begin{prop}\la{zero}
      Suppose $(c_1^*, c_2^*, u^*)$ is a smooth solution of the steady state NPS system \eqref{s}-\eqref{S}, satisfying boundary conditions \eqref{b}-\eqref{B}. Then
      \begin{enumerate}[(i)]
          \item $\|c_1^*\|_{L^\infty}+\|c_2^*\|_{L^\infty}< A_\gamma$
          \item $\|\na \Phi^*\|_{L^\infty}< B_\gamma$
      \end{enumerate}
      where $A_\gamma, B_\gamma>0$ are constants depending only on the parameters and the boundary conditions, satisfying
      \begin{align}
          A_\gamma+B_\gamma\to 0 \quad\text{as}\quad \|\gamma_1\|_{L^\infty(\pa\Omega)}+\|\gamma_2\|_{L^\infty(\pa\Omega)}+\|\na \tilde W\|_{L^\infty}\to 0
      \end{align}
      where $\tilde W$ is the harmonic extension to $\Omega$ of $W$. 
 \end{prop}

\begin{proof}
The proof of (i) is contained in Remark 6 of \cite{NPS}. In particular, we may take $$A_\gamma = \max(\|\gamma_1\|_{L^\infty(\pa\Omega)}, \|\gamma_2\|_{L^\infty(\pa\Omega)}),$$ 
Then, from elliptic regularity and Sobolev embeddings applied to the Poisson equation
\begin{align*}
    -\epsilon\D (\Phi^*-\tilde W)=c_1^*-c_2^*
\end{align*}
we obtain
\begin{align*}
    &\|\na(\Phi^*-\tilde W)\|_{L^\infty}\le C\|\Phi^*-\tilde W\|_{W^{2,4}}\le C\|c_1^*-c_2^*\|_{L^4}\\
    &\Rightarrow \|\na\Phi^*\|_{L^\infty}\le C(\|c_1^*\|_{L^4}+\|c_2^*\|_{L^4})+\|\na\tilde W\|_{L^\infty}
\end{align*}
from which (ii) follows, using (i). 
\end{proof}

Now, given that our goal is to study the stability of nonequilibrium steady state solutions with nonzero fluid flow, we ask under what conditions the solutions obtained via Theorem \ref{stex} are known to satisfy $u^*\not\equiv 0$. In fact, there is a sufficient condition that guarantees this:

\begin{thm}\la{boundarythm}
    (Theorem 2 of \cite{NPS}) Suppose $(c_1^*, c_2^*, u^*)$ is a smooth solution of the steady state NPS system \eqref{s}-\eqref{S}, satisfying boundary conditions \eqref{b}-\eqref{B}. Suppose in addition that the boundary conditions satisfy
    \begin{align}
        \int_{\pa\Omega}(\gamma_1-\gamma_2)(n\times\nabla) W\,dS\neq 0\quad \text{or}\quad \int_{\pa\Omega}W(n\times\nabla)(\gamma_1-\gamma_2)\,dS\neq 0\la{bints}
    \end{align}
    where $n$ is the outward unit normal along $\pa \Omega$. Then $u^*\not\equiv 0$.
\end{thm}
 \begin{rem}
     We note that each component of $n\times\nabla$ is a vector field tangent to $\pa\Omega$. Thus the integrals \eqref{bints} are well defined and can be computed with just knowledge of the prescribed values of $c_i$ and $\Phi$ on $\pa\Omega$. In particular, we observe that \eqref{bints} will be satisfied for almost any choice of ``rough'' or oscillatory boundary conditions.
 \end{rem}

Lastly, for there to be any hope of stability for nonequilibrium steady states, it must at least be the case that all time dependent solutions asymptotically obey the same bounds satisfied by the corresponding steady states (Proposition \ref{zero}); the following theorem verifies this fact.

\begin{thm}\la{alphaa} (Theorem 3 of \cite{NPS}) Let $(c_1, c_2, u)$ be a smooth solution of the NPS system \eqref{np}-\eqref{nse} with boundary conditions \eqref{BC}. For any $\alpha>0$, there exists $T_\alpha$ depending additionally on parameters, and initial and boundary conditions such that for all $t\ge T_\alpha$, 
    \begin{align}\la{hmm}
        \|c_1(t)\|_{L^\infty}+\|c_2(t)\|_{L^\infty}\le A_\gamma + \alpha
    \end{align}
    where $A_\gamma$ is as in Proposition \ref{zero}.
\end{thm}

\section{Stability of Weak Electrokinetic Flow}
We remark that Theorem \ref{alphaa} brings time dependent solutions close to the steady states, in a very loose sense. Our main result below shows that for sufficiently \textit{small} steady state solutions, \eqref{hmm} is in fact close enough to guarantee that time dependent solutions asymptotically converge to them.
\begin{thm}
    There exists $\delta>0$, depending only on parameters, such that if $\gamma_1, \gamma_2, W\in C^\infty(\pa\Omega)$ with
    \begin{align}
    \|\gamma_1\|_{L^\infty(\pa\Omega)}+\|\gamma_2\|_{L^\infty(\pa\Omega)}+\|\na \tilde W\|_{L^\infty}<\delta\la{small}
    \end{align}
    then solutions of the NPS system \eqref{np}-\eqref{nse}, with boundary conditions 
    \be
    \bal
    {c_i}_{|\pa\Omega}&=\gamma_i>0,\quad i=1,2\\
    \Phi_{|\pa\Omega}&=W\\
    u_{|\pa\Omega}&=0\la{BB}
    \eal
    \ee
    are globally asymptotically stable. That is, for boundary conditions satisfying \eqref{small}, there exists a unique, corresponding, steady state solution $(c_1^*, c_2^*, u^*)$, and $(c_1(t), c_2(t), u(t))\to (c_1^*, c_2^*, u^*)$ in $L^2$ as $t\to\infty$, regardless of initial conditions. Furthermore, if the boundary conditions satisfy \eqref{bints}, then $u^*\not\equiv 0$.
\end{thm}

\begin{proof}
    The existence of a steady state solution $(c_1^*, c_2^*, u^*)$ is guaranteed by Theorem \ref{stex}. Then the Nernst-Planck equations \eqref{np} may be rewritten in terms of $c_1-c_1^*$ and $c_2-c_2^*$ as
    \begin{align}
        \pa_t(c_1-c_1^*)&=D_1\div(\na(c_1-c_1^*)+c_1\na(\Phi-\Phi^*)+(c_1-c_1^*)\na\Phi^*)\la{lin1}\\
        \pa_t(c_2-c_2^*)&=D_2\div(\na(c_2-c_2^*)-c_2\na(\Phi-\Phi^*)-(c_2-c_2^*)\na\Phi^*).\la{lin2}
    \end{align}
    Multiplying \eqref{lin1} by $c_1-c_1^*$ and integrating by parts, we obtain
    \be
    \bal
        \fr{1}{2}\fr{d}{dt}\|c_1-c_1^*\|_{L^2}^2+D_1\|\na(c_1-c_1^*)\|_{L^2}^2 =& -D_1\int_\Omega c_1\na(\Phi-\Phi^*)\cdot\na(c_1-c_1^*)\,dx\\
        &-D_1\int_\Omega (c_1-c_1^*)\na\Phi^*\cdot\na(c_1-c_1^*)\,dx.
    \eal
    \ee
    Then from elliptic estimates $\|c_1-c_1^*\|_{L^2}\le C\|\na (c_1-c_1^*)\|_{L^2}$ and $\|\na(\Phi-\Phi^*)\|_{L^2}\le C(\|\na (c_1-c_1^*)\|_{L^2}+\|\na (c_2-c_2^*)\|_{L^2})$ we obtain
    \be
    \bal\la{kek}
        \fr{1}{2}\fr{d}{dt}\|c_1-c_1^*\|_{L^2}^2+D_1\|\na(c_1-c_1^*)\|_{L^2}^2 \le& D_1C_\Gamma(\|c_i\|_{L^\infty}+\|\na\Phi^*\|_{L^\infty})(\|\na(c_1-c_1^*)\|^2+\|\na(c_2-c_2^*)\|_{L^2}^2)
    \eal
    \ee
    where $C_\Gamma$ depends only on parameters (excluding boundary and initial conditions). Now, referring to Proposition \ref{zero}, we choose $\delta>0$ small enough so that $B_\gamma<\fr{1}{8C_\Gamma}$. Then we choose $\alpha>0$ and choose $\delta$ smaller (if necessary) so that $A_\gamma+\alpha<\fr{1}{8C_\Gamma}$. These choices are possible by Proposition \ref{zero} and Theorem \ref{alphaa}. Then for $t$ large enough (that is, larger than $T_\alpha$, Theorem \ref{alphaa}), we have
    \be
    \bal\la{molla}
        \fr{1}{2}\fr{d}{dt}\|c_1-c_1^*\|_{L^2}^2+D_1\|\na(c_1-c_1^*)\|_{L^2}^2 \le& \fr{D_1}{4}(\|\na(c_1-c_1^*)\|_{L^2}^2+\|\na(c_2-c_2^*)\|_{L^2}^2)
    \eal
    \ee
    We obtain analogous estimates for $c_2$, and then adding to \eqref{molla}, we obtain
    \begin{align}\la{haha}
    \fr{1}{2}\fr{d}{dt}(\|c_1(t)-c_1^*\|_{L^2}^2+\|c_2(t)-c_2^*\|_{L^2}^2)+\fr{D}{2}(\|\na(c_1(t)-c_1^*)\|_{L^2}^2+\|\na(c_2(t)-c_2^*)\|_{L^2}^2) \le 0
    \end{align}
    for $D=\min\{D_1, D_2\}$ and $t\ge T_\alpha$. Then, by using the following Poincaré inequality
    \begin{align}
        \|c_i-c_i^*\|_{L^2}\le C\|\na(c_i-c_i^*)\|_{L^2}
    \end{align}
   we obtain from \eqref{haha}
    \begin{align}\la{AHHH}
        \fr{1}{2}\fr{d}{dt}(\|c_1(t)-c_1^*\|_{L^2}^2+\|c_2(t)-c_2^*\|_{L^2}^2)+\tilde C(\|c_1(t)-c_1^*\|_{L^2}^2+\|c_2(t)-c_2^*\|_{L^2}^2) \le 0
    \end{align}
   which holds for large $t$, $t\ge T_\alpha$. 

    As for the fluid equation, \eqref{nse} may be rewritten in terms of $u-u^*$,
    \begin{align}
        \pa_t(u-u^*)-\nu\D(u-u^*)+\na(p-p^*)= -K\rho\na(\Phi-\Phi^*)-K(\rho-\rho^*)\na\Phi^*.
    \end{align}
    Taking the inner product of the above with $u-u^*$ and integrating by parts, we obtain
    \be
    \bal
        \fr{1}{2}\fr{d}{dt}\|u-u^*\|_{L^2}^2+\nu\|\na(u-u^*)\|_{L^2}^2=&-K\int_\Omega\rho\na(\Phi-\Phi^*)\cdot(u-u^*)\,dx\\
        &-K\int_\Omega(\rho-\rho^*)\na\Phi^*\cdot(u-u^*)\,dx.
    \eal
    \ee
    Then applying elliptic estimates as in \eqref{kek}, we obtain
    \be
    \bal
        \fr{1}{2}\fr{d}{dt}\|u-u^*\|_{L^2}^2+\nu\|\na(u-u^*)\|_{L^2}^2\le &C(\|c_1\|_{L^\infty}+\|c_2\|_{L^\infty}+\|\na\Phi^*\|_{L^\infty})\times\\
        &(\|c_1-c_1^*\|_{L^2}+\|c_2-c_2^*\|_{L^2})\|\na(u-u^*)\|_{L^2}
    \eal
    \ee
    Then, again ensuring that our choices of $\delta$ and $\alpha$ are small enough, we conclude that for $t\ge T_\alpha$, we have
    \be
    \bal
        \fr{1}{2}\fr{d}{dt}\|u-u^*\|_{L^2}^2+\fr{\nu}{2}\|\na(u-u^*)\|_{L^2}^2\le &\fr{\tilde C}{2}(\|c_1-c_1^*\|_{L^2}+\|c_2-c_2^*\|_{L^2}).
    \eal
    \ee
    where $\tilde C$ is from \eqref{AHHH}. Adding this estimate to \eqref{AHHH}, and using the Poincaré estimate $\|u-u^*\|_{L^2}\le C\|\na(u-u^*)\|_{L^2}$, we finally obtain
    \begin{align}
        \fr{d}{dt}\mathcal{E}(t)+r\mathcal{E}(t)\le 0
    \end{align}
    for all $t\ge T_\alpha$, for 
    \begin{align}
        \mathcal{E}(t)=\|c_1(t)-c_1^*\|_{L^2}^2+\|c_2(t)-c_2^*\|_{L^2}^2+\|u(t)-u^*\|_{L^2}
    \end{align}
    and $r>0$ depending only on parameters. Then a Grönwall estimate implies
    \be
    \bal
        \mathcal{E}(t) \le \mathcal{E}(T_\alpha)e^{-r(t-T_\alpha)}
    \eal
    \ee
    for all $t\ge T_\alpha$. In particular,
    \begin{align}
        \|c_1(t)-c_1^*\|_{L^2}^2+\|c_2(t)-c_2^*\|_{L^2}^2+\|u(t)-u^*\|_{L^2}\to 0
    \end{align}
    as $t\to\infty$, as was to be shown.
    
    Lastly, we note that, a priori, the steady state solution $(c_1^*, c_2^*, u^*)$ whose stability we proved, need not have been unique. However, our proof of global stability serves as a proof that in fact, steady state solutions are unique, given small boundary data, in the sense of \eqref{small}. A direct proof of uniqueness, without appealing to the time dependent NPS system, would follow from similar estimates as above, applied to the difference between two steady state solutions.
\end{proof}

{\bf{Acknowledgment.}} The author was partially supported by an AMS-Simons Travel Grant.

{\bf{Conflicts of Interest.}} The author states that there are no conflicts of interest.

{\bf{Data Availability.}} No data are associated with this article.

\end{document}